\documentclass{article}
\usepackage{latexsym,amssymb,amsmath,amsthm}
\newtheorem{theo}{Theorem}[section]

\newtheorem{lemm}[theo]{Lemma}
\newtheorem{prop}[theo]{Proposition}

\newtheorem{rema}[theo]{Remark}

\author{I.D. Chipchakov}
\title{Henselian valued quasilocal fields with totally
indivisible value groups}
\date{}
\begin{document}
\maketitle

\begin{abstract}
This paper characterizes the quasilocal fields from the class of
Henselian valued fields with totally indivisible value groups, which
possess finite separable extensions of nontrivial defect. We show
that, for any prime number $q$, a divisible subgroup $T$ in the 
multiplicative group of complex roots of unity is realizable as the 
Brauer group of such a quasilocal field of residual characteristic 
$q$ unless $q = 2$ and the $2$-component of $T$ is trivial.
\end{abstract}

\vskip0.4truecm
{\it 2010 Mathematics Subject Classification:} 12J10, 16K50 
(Primary); 12F10

{\it Keywords:} Quasilocal field, Brauer group, Henselian valuation, 
immediate extension, norm-inertial/quasiinertial extension, totally 
indivisible value group

\vskip0.75truecm
\centerline{\bf Introduction}
\par
\vskip.57truecm
This paper is a continuation of \cite{Ch1}. Let $K$ be a field, $K 
^{\ast }$ its multiplicative group, $K _{\rm sep}$ a separable 
closure of $K$, $\mathcal{G}_{K} = \mathcal{G}(K _{\rm sep}/K)$ the 
absolute Galois group of $K$, and for any prime number $p$, let 
cd$_{p}(\mathcal{G}_{K})$ be the cohomological $p$-dimension of 
$\mathcal{G}_{K}$, $K(p)$ the maximal $p$-extension of $K$ in $K 
_{\rm sep}$, and $r(p)_{K}$ the rank of the Galois group 
$\mathcal{G}(K(p)/K)$ as a pro-$p$-group ($r(p)_{K} = 0$ in case 
$K(p) = K$). We say that $K$ is primarily quasilocal (abbr, PQL), if 
every cyclic extension $F$ of $K$ is embeddable as a subalgebra in 
each central division $K$-algebra $D$ of Schur index ind$(D)$ 
divisible by the degree $[F\colon K]$; $K$ is called quasilocal, if 
its finite extensions are PQL-fields. The class of quasilocal fields 
includes the one of local fields and contains $p$-adically closed 
fields and Henselian discrete valued fields with quasifinite residue 
fields (cf. \cite{S2}, Ch. XIII, Sect. 3, \cite{PR}, Theorem~3.1 and 
Lemma~2.9, and \cite{Ch3}, Proposition~6.4). The quasilocal property 
has been fully characterized by \cite{Ch1}, Theorem~2.1, in the 
class of Henselian (valued) fields with totally indivisible value 
groups, whose finite separable extensions are defectless. Other 
examples of quasilocal fields, mostly, of nonarithmetic nature (from 
the perspective of \cite{Ch3}, (1.2), (1.3) and Corollary~5.3), can 
be found in \cite{Ch5}.
\par
\medskip
The present paper proves the existence of quasilocal Henselian fields 
with totally indivisible value groups, that admit defectful finite 
separable extensions. It describes, up-to an isomorphism, the abelian 
torsion groups that can be realized as Brauer groups of such fields.
\par
\medskip
\section{\bf Statement of the main result}
\par
\medskip
A nontrivial (Krull) valuation $v$ of a field $K$ is said to be 
Henselian, if it is uniquely, up-to an equivalence, extendable to a 
valuation $v _{L}$ on each algebraic field extension $L/K$. This is 
the case if and only if the valuation ring $O _{v}(K) = \{a \in 
K\colon \ v(a) \ge 0\}$ is Henselian with respect to its (unique) 
maximal ideal $M _{v}(K) = \{a \in K\colon \ v(a) > 0\}$ (see 
(2.1)). Denote by $v(K)$ the value group and by $\widehat K$ the 
residue field of $(K, v)$. We say that $v(K)$ is totally indivisible, 
if it is $p$-indivisible, i.e. $v(K) \neq pv(K)$, for every $p \in 
\mathbb P$. As a beginning of our considerations, we introduce the 
notions of a norm-inertial extension and of a quasiinertial 
extension, as follows:
\par
\medskip
{\bf Definitions.} Let $(K, v)$ be a Henselian field with 
char$(\widehat K) = p > 0$, $M$ a finite extension of $K$ in $K(p)$, 
and $v _{M}$ a valuation of $M$ extending $v$. The extension $M/K$ is 
called norm-inertial, if the norm group $N(M/K)$ contains all $\theta 
\in K ^{\ast }$ with $v(\theta - 1) > 0$. We say that $M/K$ is 
quasiinertial, if the valuation ring $O _{v _{M}}(M)$ consists of 
those elements $\delta \in M ^{\ast }$, for which the trace Tr$_{K} 
^{M}(\delta \mu )$ has value $\ge 0$, for each $\mu \in O _{v}(M)$.
\par
\medskip
Our next result and \cite{Ch1}, Theorem~2.1, give a formally complete 
characterization of quasilocal Henselian fields with totally 
indivisible value groups, and attract interest in the algebraic 
nature of immediate norm-inertial extensions:
\par
\medskip
\begin{prop}
\label{prop1.1}
Let $(K, v)$ be a Henselian field admitting a finite extension in $K
_{\rm sep}$ of nontrivial defect, and for each prime $p$, let $G
_{p}$ be a Sylow pro-$p$-subgroup of $\mathcal{G}_{K}$ and $K _{p}$
the fixed field of $G _{p}$. Suppose that {\rm char}$(\widehat K) =
q$ and $v(K) \neq pv(K)$ whenever $G _{p} \neq \{1\}$. Then $K$ is
quasilocal if and only if it satisfies the following:
\par
{\rm (a)} The quotient group $v(K)/qv(K)$ is of order $q$, $\widehat
K$ is perfect, {\rm cd}$_{q}(\mathcal{G}_{\widehat K}) = 0$, and $K
_{q}$ has an immediate $\mathbb Z _{q}$-extension $Y$ in $K _{\rm
sep}$, such that every finite extension $L _{q}$ of $K _{q}$ in $K
_{\rm sep}$ with $L _{q} \cap Y = K _{q}$ is totally ramified; in
addition, finite extensions of $K _{q}$ in $Y$ are norm-inertial;
\par
{\rm (b)} $r(p)_{K _{p}} \le 2$, for each prime $p \neq q$.
\end{prop}

\medskip
Proposition \ref{prop1.1} has been proved as \cite{Ch3}, 
Proposition~6.1. The main result of the present paper (stated 
without proof in \cite{Ch3}, Sect. 6) provides series of examples of 
quasilocal Henselian real-valued fields satisfying the conditions of 
this proposition. Before stating it, note that the assumptions of 
Proposition \ref{prop1.1} ensure that $K$ is a nonreal field 
\cite{La}, Theorem~3.16, which implies that the Brauer group Br$(K)$ 
is divisible whenever $K$ is quasilocal (cf. \cite{Ch2}, 
Theorem~3.1). At the same time, in the quasilocal case, by 
\cite{Ch3}, Theorem~1.1, Br$(K)$ is embeddable as a subgroup in the 
quotient group $\mathbb Q/\mathbb Z$ of the additive group of 
rational numbers by the subgroup of integers. Conversely, divisible 
subgroups of $\mathbb Q/\mathbb Z$ are realizable as Brauer groups 
of quasilocal Henselian fields of the type studied in \cite{Ch1} 
(see (2.6) and \cite{TY}, Proposition~2.2). These observations 
attract interest in the description of the isomorphism classes of 
Brauer groups of the quasilocal Henselian fields admissible by 
Proposition \ref{prop1.1}. Our main result in this direction is 
contained in the following theorem:
\par
\medskip
\begin{theo}
\label{theo1.2}
Let $(\Phi , \omega )$ be a Henselian discrete valued field with
$\widehat \Phi $ quasifinite and {\rm char}$(\widehat K) = q \neq 
0$, and let $T$ be a divisible subgroup of $\mathbb Q/\mathbb Z$ 
with a nontrivial $q$-component $T _{q}$. Then there is a quasilocal
Henselian field $(K, v)$ such that:
\par
{\rm (a)} {\rm Br}$(K)$ is isomorphic to $T$, $K/\Phi $ is a field
extension of transcendency degree $1$ and $v$ is a prolongation of
$\omega $;
\par
{\rm (b)} $v(K)$ is a totally indivisible Archimedean group, 
$\widehat K/\widehat \Phi $ is an algebraic extension, and $K$
possesses an immediate quasiinertial $\mathbb Z _{q}$-extension $I 
_{\infty }$.
\end{theo}
\par
\medskip
Theorem \ref{theo1.2} is proved in Section 4. Its proof relies on
the characterization of quasiinertial Galois extensions given in
Section 2, and on their relations with norm-inertial Galois 
extensions (these results, in a special case, also play a role in 
the proof of Proposition \ref{prop1.1}, see (2.8), Lemma 
\ref{lemm2.1} and \cite{Ch2}, (3.4)). In addition, we use the 
quasilocal property of $\Phi $ and an easily applicable criterion 
for the fulfillment of the quasiinertial condition, presented in 
Section 3. In Section 5 we obtain similarly that if $q > 2$, then 
divisible subgroups $T \le \mathbb Q/\mathbb Z$ with $T _{q} = 
\{0\}$ are also realizable as Brauer groups of quasilocal fields 
admissible by Proposition \ref{prop1.1}. The case of $q = 2$ is 
exceptional - then Br$(K) _{2}$ is a quasicyclic $2$-group whenever 
$K$ is a quasilocal field satisfying the conditions of Proposition 
\ref{prop1.1} (see Proposition \ref{prop5.1}).
\par
Note that Brauer groups of quasilocal fields $E$ have influence on a
wide spectrum of their algebraic properties. This includes the
structure of the continuous character groups of
$\mathcal{G}(E(p)/E)$, $p \in \mathbb P$ \cite{Ch2}, II, Lemmas~2.3
and 3.3, cohomological properties of $\mathcal{G}(E(p)/E)$ and the
Sylow pro-$p$-subgroups of $\mathcal{G}_{E}$ \cite{Ch2}, I,
Theorem~8.1, and \cite{Ch4}, Sect. 5, finite abelian extensions of
$E$ and their norm groups \cite{Ch4} (concerning nonabelian Galois
extensions of $E$, see \cite{Ch5}). Therefore, the description of
Br$(E)$ is a major objective of the study of $E$, and the present
research can be viewed as the final step towards a really complete
characterization of quasilocal Henselian fields with totally
indivisible value groups.
\par
\medskip
The basic notation, terminology and conventions kept in this paper
are standard and essentially the same as in \cite{Ch2}, I,
\cite{Ch3} and \cite{Ch4}. Preliminaries on Henselian valuations
used in the sequel are included in Section 2. Throughout, Brauer and
value groups are additively presented, Galois groups are viewed as
profinite with respect to the Krull topology, and by a profinite
group homomorphism, we mean a continuous one. We write $\mathbb P$
for the set of prime numbers, and for each $p \in \mathbb P$,
$\mathbb Z _{p}$ denotes the additive group of $p$-adic integers and
$\mathbb Z(p ^{\infty })$ is the quasicyclic $p$-group. For any
profinite group $G$, cd$(G)$ is the cohomological dimension of $G$,
and cd$_{p}(G)$, $p \in \mathbb P$, are its cohomological
$p$-dimensions. Given a field $E$, Br$(E) _{p}$ is the $p$-component
of the Brauer group Br$(E)$, and $_{p} {\rm Br}(E) = \{\delta \in
{\rm Br}(E)\colon \ p\delta = 0\}$, where $p \in \mathbb P$, $P(E) =
\{p \in \mathbb P\colon \ E(p) \neq E\}$, and $\Pi (E) = \{p \in
\mathbb P\colon \ {\rm cd}_{p}(\mathcal{G}_{E}) > 0\}$. We write
$s(E)$ for the class of finite-dimensional central simple
$E$-algebras, $d(E)$ stands for the class of division algebras $D
\in s(E)$, and for each $A \in s(E)$, $[A]$ is the similarity class
of $A$ in Br$(E)$. For any field extension $E ^{\prime }/E$, $I(E
^{\prime }/E)$ denotes the set of its intermediate fields. By a
$\mathbb Z _{p}$-extension of $E$, for some $p \in \mathbb P$, we
mean a Galois extension $E _{\infty }/E$ with a Galois group
$\mathcal{G}(E _{\infty }/E) \cong \mathbb Z _{p}$. The field $E$ is
called $p$-quasilocal, if Br$(E) _{p} = \{0\}$, or $p \notin P(E)$,
or every degree $p$ extension of $E$ in $E(p)$ embeds as an
$E$-subalgebra in each $\Delta _{p} \in d(E)$ of index $p$. Note
that $E$ is PQL if and only if it is $p$-quasilocal, for each $p \in
P(E)$ (cf. \cite{P}, Sects. 13.4, 14.4 and 15.3).
\par
\medskip

\section{Preliminaries on Henselian valuations and characterizations
of quasiinertial Galois extensions}
\par
\medskip
Let $(K, v)$ be a (nontrivially) valued field, $K _{v}$ a completion 
of $K$ relative to the topology induced by $v$, $v(K) _{0} = 
\{\gamma \in v(K)\colon \ \gamma > 0\}$ and $\nabla _{\gamma }(K) = 
\{\alpha \in K\colon \ v(\alpha - 1) \ge \gamma \}$, for each 
$\gamma \in v(K) _{0} \cup \{0\}$. It is known that $v$ is Henselian 
if and only if the following condition holds (cf. \cite{E3}, Sect. 
18.1):
\par
\medskip
(2.1) Given a polynomial $f(X) \in O _{v}(K) [X]$, and an element $a
\in O _{v}(K)$, such that $2v(f ^{\prime }(a)) < v(f(a))$, where $f
^{\prime }$ is the (formal) derivative of $f$, there is a zero $c
\in O _{v}(K)$ of $f$ satisfying the equality $v(c - a) = v(f(a)/f
^{\prime }(a))$.
\par
\medskip
The fulfillment of (2.1) ensures that the polynomial $f _{b}(X) =
f(X) + b$ has a zero in $K$ whenever $b \in K ^{\ast }$ and $v(b) >
2v(f ^{\prime }(a))$. Also, the Henselity of $v$ is inherited by $v
_{M}$, for every algebraic field extension $M/K$. When $[M\colon K]$
is finite and $M \subseteq K _{\rm sep}$, these observations,
applied to the minimal polynomial $f _{\beta }(X)$ over $K$ of a
primitive element $\beta \in O _{v _{M}}(M)$ of $M/K$, prove the
following:
\par
\medskip
(2.2) The norm group $N(M/K)$ contains every element $\alpha \in O 
_{v}(K)$, for which $v(\alpha - 1) > 2v _{M}(f _{\beta } ^{\prime 
}(\beta ))$.

\medskip
When $v$ is Henselian and $L/K$ is algebraic, $v _{L}$ is Henselian
and extends uniquely to a valuation $v _{D}$ on each $D \in d(L)$.
Denote by $\widehat D$ the residue field of $(D, v _{D})$ and put
$v(D) = v _{D}(D)$. By the Ostrowski-Draxl theorem \cite{Dr2},
$[D\colon K]$, $[\widehat D\colon \widehat K]$ and the ramification
index $e(D/K)$ are related as follows:
\par
\medskip
(2.3) $[D\colon K]$ is divisible by $[\widehat D\colon \widehat
K]e(D/K)$ and the defect
\par\noindent
$d(D/K) = [D\colon K]/([\widehat D\colon \widehat K]e(D/K))$ is not
divisible by any $p \in \mathbb P$, $p \neq {\rm char}(\widehat K)$.
\par
\medskip
The $K$-algebra $D$ is said to be defectless, if $d(D/K) = 1$, i.e.
$[D\colon K] = [\widehat D\colon \widehat K]e(D/K)$; it is called
immediate, if $\widehat D = \widehat K$ and $e(D/K) = 1$. We say
that $D/K$ is totally ramified, if $e(D/K) = [D\colon K]$. When
$v(K) \neq pv(K)$, for a given $p \in \mathbb P$, $(K, v)$ is
subject to the following alternative (see \cite{Ch6},
Corollary~6.5):
\par
\medskip
(2.4) (i) $K$ has a totally ramified proper extension in $K(p)$;
\par
(ii) char$(K) = 0$, $K$ does not contain a primitive $p$-th root of
unity and the minimal isolated subgroup of $v(K)$ containing $v(p)$
is $p$-divisible.
\par
\medskip
A finite extension $R$ of $K$ is said to be inertial, if $[R\colon
K] = [\widehat R\colon \widehat K]$ and $\widehat R$ is separable
over $\widehat K$; $R/K$ is called tamely ramified, if $\widehat
R/\widehat K$ is separable and $e(R/K)$ is not divisible by
char$(\widehat K)$. It is well-known that the compositum $K _{\rm
ur}$ of inertial extensions of $K$ in $K _{\rm sep}$ is a Galois
extension of $K$, and so is the compositum $K _{\rm tr}$ of tamely
ramified extensions of $K$ in $K _{\rm sep}$. Note also that $K
_{\rm ur}$ and $K _{\rm tr}$ have the following properties:
\par
\medskip
(2.5) (i) $v(K _{\rm ur}) = v(K)$ and finite extensions of $K$ in $K
_{\rm ur}$ are inertial;
\par
(ii) $K _{\rm tr}$ contains a primitive $m$-th root of unity, for 
each $m \in \mathbb N$ not divisible by char$(\widehat K)$, finite 
extensions of $K$ in $K _{\rm tr}$ are tamely ramified, and $v(K 
_{\rm tr}) = pv(K _{\rm tr})$, for every $p \in \mathbb P$ different 
from char$(\widehat K)$;
\par
(iii) $\widehat K _{\rm ur}$ is $\widehat K$-isomorphic to $\widehat
K _{\rm sep}$, $\mathcal{G}(K _{\rm ur}/K) \cong
\mathcal{G}_{\widehat K}$, and the natural mapping of $I(K _{\rm
ur}/K)$ into $I(\widehat K _{\rm sep}/\widehat K)$ is bijective.
\par
\medskip
When $(K, v)$ is a local field, $T$ is a divisible subgroup of
$\mathbb Q/\mathbb Z$, and $S(T) = \{p \in \mathbb P\colon \ T _{p}
\neq \{0\}\}$, there exists $K _{T} \in I(K _{\rm ur}/K)$, such that
$\mathcal{G}(K _{\rm ur}/K _{T})$ is isomorphic to the topological
group product $\prod _{p \in S(T)} \mathbb Z _{p}$. In other words,
$T$ is isomorphic to the continuous character group of $\mathcal{G}(K
_{\rm ur}/K _{T})$. Since $v(K _{T}) = v(K)$ and Br$(\widehat K _{T})
= \{0\}$, this enables one to deduce from Witt's theorem (cf.
\cite{Wad}, (3.10)) that Br$(K _{T}) \cong T$. It is therefore clear
from \cite{Ch3}, Corollary~5.3, that
\par
\medskip
(2.6) An abelian torsion group is realizable as the Brauer group of
a quasilocal Henselian field with a totally indivisible value group
and defectless finite separable extensions if and only if it is
divisible and embeddable in $\mathbb Q/\mathbb Z$.
\par
\medskip
Let now $(K, v)$ be a Henselian field with char$(\widehat K) = p
> 0$, and let $M \in I(K(p)/K)$ be a finite extension of $K$. Then:
\par
\medskip
(2.7) $\nabla _{0}(M)$ equals the pre-image of $\nabla _{0}(K)$,
under the norm map $N _{K}^{M}$, provided that $\widehat M =
\widehat K$; in this case, $\varphi (\mu )\mu ^{-1} \in \nabla
_{0}(M)$ whenever $\mu \in M ^{\ast }$ and $\varphi $ is a
$K$-automorphism of $M$.
\par
\medskip
With notation being as above, put $\delta _{M/K}(\mu ) = v _{M}(f
_{\mu }^{\prime }(\mu ))$, for each primitive element $\mu $ of
$M/K$, where $f _{\mu }^{\prime }$ is the derivative of the minimal
(monic) polynomial $f _{\mu }$ of $\mu $ over $K$. Clearly,
$[M\colon K]\delta _{M/K}(\mu ) = v(d _{\mu })$, $d _{\mu }$ being
the discriminant of $f _{\mu }$. This fact and the the following 
lemma will be used in the sequel.
\par
\medskip
\begin{lemm}
\label{lemm2.1}
Let $(K, v)$ be a Henselian field with {\rm char}$(\widehat K) = p >
0$, $M$ a finite Galois extension of $K$ in $K(p)$, and for each
primitive element $\mu $ of $M/K$ lying in $O _{v}(M)$, let $f _{\mu
}(X)$ be the minimal polynomial of $\mu $ over $K$, and $\delta
_{M/K}(\mu ) = v _{M}(f _{\mu } ^{\prime }(\mu ))$. Then $M/K$ is
quasiinertial if and only if any of the following three equivalent
conditions is fulfilled:
\par
{\rm (a)} For each $\gamma \in v(K) _{0}$, there exists $\lambda
_{\gamma } \in O _{v}(K)$ with $v({\rm Tr}_{K} ^{M} (\lambda
_{\gamma })) < \gamma $;
\par
{\rm (b)} For each $\gamma ^{\prime } \in v(M) _{0}$, $O _{v}(M)$
contains a primitive element $\mu _{\gamma '}$ of $M/K$ satisfying
the inequality $\delta _{M/K}(\mu _{\gamma '}) < \gamma ^{\prime }$;
\par
{\rm (c)} There exists $L \in I(M/K)$, such that $L/K$ and $M/L$ are
quasiinertial;
\par
{\rm (d)} For any $\gamma \in v(K) _{0}$, there is $\beta _{\gamma }
\in O _{v}(M)$, such that $v _{M}(\varphi (\beta _{\gamma }) - \beta 
_{\gamma }) < \gamma $, for every $\varphi \in \mathcal{G}(M/K)$ 
different from $1$.
\par\noindent
When $M/K$ is quasiinertial, so are $M/M _{0}$ and $M _{0}/K$, for
every $M _{0} \in I(M/K)$.
\end{lemm}

\medskip
\begin{proof} The concluding assertion of the lemma follows from the 
claimed equivalence of condition (a) and the one that $M/K$ is 
quasiinertial, together with the inequalities $v _{M}(y) \le v({\rm 
Tr}_{K}^{M}(y))$, $y \in O _{v}(M)$, and the transitivity of traces 
in towers of finite separable extensions (cf. \cite{L}, Ch. VIII, 
Sect. 5). When condition (c) of Lemma \ref{lemm2.1} holds, the 
assertion that $M/K$ is quasiinertial is standardly proved by 
assuming the opposite, using again trace transitivity (specifically, 
the equality Tr$_{K}^{M} = {\rm Tr}_{K}^{L} \circ {\rm Tr}_{L}^{M}$) 
and the $L$-linearity of Tr$_{L}^{M}$. Thus (c) turns out to be 
equivalent to the assumption that $M/K$ is quasiinertial. Let $r \in 
O _{v}(M)$ be a primitive element of $M/K$. It is easily obtained 
(by applying basic linear algebra, including Cramer's rule and 
Vandermonde's determinant) that if $r ^{\prime } \in O _{v}(M) 
\setminus \{0\}$ and Tr$_{K} ^{M}(r ^{\prime -1}r ^{j-1}) \in O 
_{v}(K)$, $j = 1, \dots , [M\colon K]$, then $2v _{M}(r ^{\prime }) 
\le v(d _{r})$. Hence, the validity of condition (b) of Lemma 
\ref{lemm2.1} ensures that $M/K$ is quasiinertial. As to condition 
(a), it is satisfied in case $M/K$ is quasiinertial (because if $a 
\in M _{v}(K) \setminus \{0\}$ and $a ^{\prime } \in O _{v}(M)$, 
then Tr$_{K}^{M}(a ^{-1}a ^{\prime }) \in O _{v}(K)$ if and only if 
$v(a) \le v({\rm Tr}_{K} ^{M}(a ^{\prime }))$). These observations 
can be summarized by saying that (b)$\to $(c)$\to $(a). We prove 
that (a)$\to $(d). Note here that if $v(K)_{0}$ contains a minimal 
element, then $M/K$ is inertial if and only if some of conditions 
(a), (b), (c), (d) is satisfied. Therefore, it suffices to prove 
that (a)$\to $(d)$\to $(b) in case $v(K)_{0}$ does not contain a 
minimal element. We first prove that (a)$\to $(d). Assume that 
condition (a) holds, $[M\colon K] = p ^{n}$ and $\alpha $ is an 
element of $O _{v}(M)$, such that $v({\rm Tr}_{K}^{M}(\alpha )) < 
v(p)$. It is easily verified that $\alpha $ is a primitive element 
of $M/K$. Let $\alpha _{u}$, $u = 1, \dots , [M\colon K]$, be the 
roots in $M$ of the minimal polynomial $f _{\alpha }$ of $\alpha $ 
over $K$. We prove the validity of condition (d) by showing that $v 
_{M'}(\alpha _{u'} - \alpha _{u''}) \le v({\rm Tr}_{K} ^{M}(\alpha 
))$, for $1 \le u ^{\prime } < u ^{\prime \prime } \le p ^{n}$. 
Suppose first that $[M\colon K] = p$ and $\varphi $ is a generator 
of $\mathcal{G}(M/K)$. Then $v _{M}(\varphi ^{\nu }(\alpha ) - 
\alpha ) = v _{M}(\varphi (\alpha ) - \alpha )$, for $\nu = 1, \dots 
, p - 1$. As $\alpha \in O _{v}(M)$ and $v({\rm Tr}_{K} ^{M}(\alpha 
)) < v(p)$, this implies the stated inequality. The proof in general 
is carried out by induction on $n$, under the inductive hypothesis 
that $n \ge 2$, and for some $K ^{\prime } \in I(M/K)$ of degree $[K 
^{\prime }\colon K] = p$, Tr$_{K'}^{M}$ is subject to analogous 
inequalities. Since Tr$_{K} ^{M}(\alpha ) = {\rm Tr}_{K} ^{K'}({\rm 
Tr}_{K'} ^{M}(\alpha ))$, whence $v _{K'}({\rm Tr}_{K'} ^{M}(\alpha 
)) \le v({\rm Tr}_{K} ^{M}(\alpha )) < v(p)$, this yields $v 
_{M'}(\alpha _{u'} - \alpha _{u''}) \le v _{K'}({\rm Tr}_{K'} 
^{M}(\alpha ))$, provided that $u ^{\prime } \neq u ^{\prime \prime 
}$ and $\alpha _{u'}$, $\alpha _{u''}$ are conjugate over $K 
^{\prime }$. Now take indices $u ^{\prime }$ and $u ^{\prime \prime 
}$ so that $\alpha _{u'}$ and $\alpha _{u''}$ be non-conjugate over 
$K ^{\prime }$. Then $\alpha _{u''} = \psi (\alpha _{u'})$, for some 
$\psi \in \mathcal{G}(M/K)$ inducing on $K ^{\prime }$ a generator, 
say, $\psi ^{\prime }$ of $\mathcal{G}(K ^{\prime }/K)$. Denote by 
$S _{u'}$ and $S _{u''}$ the sets of roots in $M$ of the minimal 
polynomials over $K ^{\prime }$ of $\alpha _{u'}$ and $\alpha 
_{u''}$, respectively. Using the normality of $\mathcal{G}(M/K 
^{\prime })$ in $\mathcal{G}(M/K)$, one obtains that if $v 
_{M}(\alpha _{u'} - \alpha _{u''}) > v({\rm Tr}_{K} ^{M}(\alpha ))$, 
then there is a bijection $\epsilon \colon S _{u'} \to S _{u''}$, 
such that $v _{M}(\alpha _{u} - \epsilon (\alpha _{u})) > v({\rm 
Tr}_{K} ^{M}(\alpha ))$ whenever $\alpha _{u} \in S _{u'}$. Our 
conclusion, however, contradicts the inequality $v _{K'}(\psi 
^{\prime }({\rm Tr}_{K'} ^{M}(\alpha )) - {\rm Tr}_{K'} ^{M}(\alpha 
)) \le v({\rm Tr}_{K} ^{M}(\alpha ))$ and thereby proves that $v 
_{M}(\alpha _{u'} - \alpha _{u''}) \le v({\rm Tr}_{K} ^{M}(\alpha 
))$. Thus the implication (a)$\to $(d) becomes obvious, and since 
(b)$\to $(c)$\to $(a), it remains to be seen that (d)$\to $(b). The 
assertion is evident, if the intersection $V$ of the nontrivial 
isolated subgroups of $v(K)$ is trivial. Suppose now that $V \neq 
\{0\}$. This means that $V$ is a minimal isolated subgroup of 
$v(K)$. Hence, $V$ is Archimedean, and by H\"{o}elder's theorem (cf. 
\cite{E3}, Theorem~2.5.2), it is isomorphic to an ordered subgroup 
of the additive group $\mathbb R$ of real numbers. Identifying $V$ 
with its isomorphic copy in $\mathbb R$, and taking into account 
that $v(K) _{0}$ does not contain a minimal element, one concludes 
that, for each $h \in V \cap v(K) _{0}$, there exist $h _{m} \in V 
\cap v(K) _{0}$, $m \in \mathbb N$, such that $mh _{m} < h$, for 
each index $m$. This observation completes the proof of implication 
(d)$\to $(b), and of Lemma \ref{lemm2.1}.
\end{proof}
\par
\medskip\noindent
Assuming again that $(K, v)$ is a Henselian field with char$(\widehat 
K) = p \neq 0$, we say that a field $I _{\infty } \in I(K(p)/K)$ is 
said to be a norm-inertial extension of $K$, if finite extensions of 
$K$ in $I _{\infty }$ are norm-inertial. The extension $I _{\infty 
}/K$ is called quasiinertial, if so are finite extensions of $K$ in 
$I _{\infty }$. When $I _{\infty }/K$ is Galois, we show that the 
two notions are related as follows (see \cite{Ch2}, (3.4), for a 
very concise proof in the special case where $I _{\infty }/K$ is a 
$\mathbb Z _{p}$-extension):
\par
\medskip
(2.8) (i) $I _{\infty }/K$ is norm-inertial, provided that it is
quasiinertial;
\par
(ii) If $I _{\infty }/K$ is an immediate norm-inertial extension and
$H \neq pH$ whenever $H \neq \{0\}$ and $H$ is an isolated subgroup 
of $v(K)$, then $I _{\infty }/K$ is quasiinertial; when this holds, 
$I _{\infty }/I$ is quasiinertial, for every $I \in I(I _{\infty 
}/K)$.
\par
\medskip\noindent
Inertial extensions of $K$ in $K(p)$ are obviously norm-inertial, so 
it is sufficient to prove (2.8) under the extra hypothesis that $v(K) 
_{0}$ does not contain a minimal element. Since quasiinertial finite 
Galois extensions of $K$ satisfy condition (b) of Lemma 
\ref{lemm2.1}, this enables one to deduce (2.8) (i) from (2.2) (by 
the method of proving implication (d)$\to $(b) of the lemma). The 
latter assertion of (2.8) (ii) is implied by the former one and 
Lemma \ref{lemm2.1}. We turn to the proof of the former part of 
(2.8) (ii). We first show that, for each $\gamma \in v(K) _{0}$, 
there exists $\gamma ^{\prime } \in v(K) _{0}$ less than $\gamma $ 
and not lying in $pv(K)$. The assertion is obvious, if the subgroup 
$V \le v(K)$ defined in the proof of Lemma \ref{lemm2.1} is trivial, 
so we assume that $V \neq \{0\}$. This ensures that $V$ embeds in 
$\mathbb R$ as an ordered subgroup, and it follows from our extra 
hypothesis on $v(K) _{0}$ that $V \cap v(K) _{0}$ does not contain a 
minimal element. Identifying $V$ with its isomorphic copy in 
$\mathbb R$, one also sees that $pV$ is dense in $\mathbb R$. These 
observations imply the existence of $\gamma ^{\prime } \in V$ with 
the required properties. Let $I$ be a finite Galois extension of $K$ 
in $I _{\infty }$, and let $\theta $ be an element of $\nabla 
_{0}(I)$, such that $N _{K}^{I}(\theta ) = 1 + \theta _{0}$, 
$v(\theta _{0}) < v(p)$ and $v(\theta _{0}) \notin pv(K)$. It is 
easily verified that $1 + \theta _{0} \notin K ^{\ast p}$, and 
therefore, $\theta $ is a primitive element of $I/K$. Denote by $f 
_{\theta }$ the minimal polynomial of $\theta $ over $K$. We show 
that $v _{I}(\theta - \theta ^{\prime }) \le v(\theta _{0})$, 
provided that $\theta ^{\prime } \in I$, $\theta ^{\prime } \neq 
\theta $ and $f _{\theta }(\theta ^{\prime }) = 0$. Assuming the 
opposite, one concludes that there exist a nontrivial cyclic 
subgroup $G \le \mathcal{G}(I/K)$ and some $\bar \gamma \in v(K)$, 
such that $v(\theta _{0}) < \bar \gamma < v(p)$ and $v _{I}(\theta - 
\psi (\theta )) \ge \bar \gamma $, for every $\psi \in G$. This 
implies $v _{I}(N _{J}^{I}(\theta ) - 1 - (\theta - 1) ^{p ^{h}}) 
\ge \bar \gamma $, where $J$ is the fixed field of $G _{0}$ and $p 
^{h}$ is the order of $H$. Thus it turns out that $v _{I}(N _{K} 
^{I}(\theta ) - 1 - (\tilde \theta - 1) ^{p}) \ge \bar \gamma $, for 
a suitably chosen $\tilde \theta \in \nabla _{0}(I)$. As $N _{K} 
^{I}(\theta ) = 1 + \theta _{0}$ and $v(\theta _{0}) < \bar \gamma 
$, our conclusion requires that $v _{I}((\tilde \theta - 1) ^{p}) = 
v(\theta _{0})$. This, however, contradicts the assumptions that 
$v(I) = v(K)$ and $v(\theta _{0}) \notin pv(K)$, and so proves that 
the roots of $f _{\theta }$ satisfy the claimed inequality. In view 
of (2.7), Lemma \ref{lemm2.1} and the noted property of the set 
$v(K) _{0} \setminus pv(K)$, the obtained result implies the former 
assertion of (2.8) (ii).

\medskip
\section{Preparation for the proof of Theorem \ref{theo1.2}}

\par
\medskip
The proof of Theorem \ref{theo1.2} relies on the following two
lemmas.
\par
\medskip
\begin{lemm}
\label{lemm3.1}
Let $(E, v)$ be a Henselian field with {\rm char}$(\widehat
E) = p \neq 0$ and $v(E) = pv(E)$. Assume that $p \in P(E)$,
$r(p) _{E} \in \mathbb N$, and in case {\rm char}$(E) = 0$, $E$
contains a primitive $p$-th root of unity $\varepsilon $. Then:
\par
{\rm (a)} $\widehat E$ is perfect and {\rm Br}$(E) _{p} = \{0\}$;
\par
{\rm (b)} $\mathcal{G}(E(p)/E)$ is a free pro-$p$-group; in
particular, every cyclic extension $L$ of $E$ in $E(p)$ lies in $I(L
_{\infty }/E)$, for some $\mathbb Z_{p}$-extension $L _{\infty }/E$,
$L \subseteq E(p)$;
\par
{\rm (c)} If $E$ is perfect and $v(E) \le \mathbb R$, then finite
extensions of $E$ in $E(p)$ are quasiinertial, whence every $\mathbb
Z _{p}$-extension of $E$ is quasiinertial.
\end{lemm}

\begin{proof}
The assumption on $r(p)_{E}$ and \cite{Ch6}, Lemma~4.1, imply that
$\widehat E$ is perfect. We show that Br$(E) _{p} = \{0\}$ and
$\mathcal{G}(E(p)/E)$ is a free pro-$p$-group. When char$(E) = p$,
this is a special case of \cite{J}, Proposition~4.4.8, and
\cite{S1}, Ch. II, Proposition~2, respectively. If $\varepsilon \in
E$, the two assertions are equivalent (by Galois cohomology, see
\cite{T}, page 265, \cite{S1}, Ch. I, 4.2, and \cite{W}, page 725),
so they are contained in \cite{E1}, Proposition~3.4 (or \cite{Ch3},
Proposition 2.5). This indicates that $\mathcal{G}(E(p)/E) \cong
\mathcal{G}_{Y}$, for some field $Y$ of characteristic $p$
\cite{LvdD}, (4.8) (see also \cite{Ch1}, Remark~2.6). The obtained
result, combined with Galois theory and Witt's lemma (see
\cite{Dr1}, Sect. 15), completes the proof of Lemma \ref{lemm3.1}
(a) and (b). Since the class of free pro-$p$-groups is closed under
taking open subgroups (cf. \cite{S1}, Ch. I, 4.2 and
Proposition~14), it becomes clear from Lemma \ref{lemm2.1} that it
suffices for the proof of Lemma \ref{lemm3.1} (c) to show that every
degree $p$ extension $F$ of $E$ in $E(p)$ is quasiinertial. If $F/E$
is inertial, there is nothing to prove, so we assume that this is
not the case. As $v(E) = pv(E)$ and $[F\colon E] = p$, this means 
that $F/E$ is immediate. Let $\psi $ be a generator of 
$\mathcal{G}(F/E)$. Clearly, the claimed property of $F/E$ can be 
deduced from the following assertion:
\par
\medskip
(3.1) $F$ contains elements $\lambda _{n}$, $n \in \mathbb N$, such
that $0 < v _{F}(\lambda _{n}) < v _{F}(\psi (\lambda _{n}) -
\lambda _{n}) < 1/n$, for each index $n$.
\par
\medskip\noindent
Our objective is to prove (3.1). Suppose first that char$(E) = p$ 
and $E$ is perfect. Then the Artin-Schreier theorem implies the 
existence of a sequence $t = \{t _{n} \in M _{v}(E)\colon \ n \in 
\mathbb N\}$, such that $t _{n+1} ^{p} = t _{n} \neq 0$ and the 
polynomial $X ^{p} - X - t _{n} ^{-1}$ is irreducible over $E$ with 
a root $\xi _{n} \in F$, for each index $n$. Observing that $\xi 
_{n} ^{-1} = t _{n}\prod _{j=1} ^{p-1} (\xi _{n} + j)$ and $v _{F} 
(\xi _{n}) = p ^{-1}v(t _{n} ^{-1})$, one obtains by direct 
calculations that $v _{F}(\xi _{n} ^{-1}) = p ^{-1}v(t _{n})$ and $v 
_{F}(\psi (\xi _{n} ^{-1}) - \xi _{n} ^{-1}) = 2v _{F}(\xi _{n} 
^{-1})$. Therefore, $\nabla _{0}(F)$ contains the elements $\lambda 
_{n} = \xi _{n}\psi (\xi _{n} ^{-1})$, $n \in \mathbb N$, and $v 
_{F}(\lambda _{n} - 1) = p ^{-1}v(t _{n})$, for every index $n$. The 
obtained result proves (3.1) in the case where char$(E) = p$ and $E$ 
is perfect. Assume now that $\varepsilon \in E$. In view of Kummer 
theory and the equality $\nabla _{0}(K)K ^{\ast p} = K ^{\ast }$ 
(cf. \cite{E2}, Lemma~3.3), $F$ is generated over $K$ by a $p$-th 
root of the sum $1 + \pi $, for some $\pi \in M _{v}(K)$. We prove 
Lemma \ref{lemm3.1} (c) together with the following statement:
\par
\medskip
(3.2) There exists a sequence $\pi _{n} \in M _{v}(K)$, $n \in
\mathbb N$, such that $(1 + (\varepsilon - 1) ^{p}\pi _{n} ^{-1})K
^{\ast p} = (1 + \pi )K ^{\ast p}$ and $1/n > v(\pi _{n}) > v(\pi
_{n+1})$, for each $n \in \mathbb N$.
\par
\medskip\noindent
As $r(p)_{K} \in \mathbb N$, Kummer theory ensures the existence of
a number $d \in \mathbb R$, $d > 0$, such that the cosets $\lambda K
^{\ast p}$, $\lambda \in K ^{\ast }$, have representatives in
$\nabla _{d}$; in particular, one may assume without loss of 
generality that $v(\pi ) \ge d$. In addition, it is not difficult to 
see that, for each $n \in \mathbb N$, $M _{v}(K)$ contains elements 
$a _{n,j}\colon \ j = 1, \dots , n$, such that $v(a _{n,1} ^{p}) = 
v(\pi )$, $v(\pi - \sum _{j=1} ^{n} a _{n,j} ^{p}) > pv(a _{n,n})$ 
and $v(a _{n,j}) \ge (d/p) + v(a _{n,(j-1)})$, provided that $j \ge 
2$. Also, it is known that $\nabla _{\bar p}(K) \subset K ^{\ast p}$, 
where $\bar p = (p/(p-1))v(p)$, which implies $(1 + \pi )K ^{\ast p} 
= (\sum _{j=1} ^{n} a _{n,j} ^{p})K ^{\ast p}$, for every 
sufficiently large $n$. Note also that $\nabla _{v(p)}(K)$ contains 
the element $(1 + \sum _{j=1} ^{n} a _{n,j} ^{p})(1 - \sum _{j=1} 
^{n} a _{n,j}) ^{p}$. Thus it turns out that there exist $b _{n} \in 
M _{v}(K)$, $n \in \mathbb N$, such that $(1 + pb _{n})K ^{\ast p} = 
(1 + \pi )K ^{\ast p}$ and $v(b _{n+1} ^{p}) = v(pb _{n})$, for each 
index $n$. As $F/K$ is immediate and $v(p) = (p - 1)v(\varepsilon - 
1)$, this implies $v(b _{n}) < v(\varepsilon - 1)$, for each $n \in 
\mathbb N$, which enables one to prove that the sequence $v(b 
_{n})$, $n \in \mathbb N$, increases and converges to $v(\varepsilon 
- 1)$. Hence, $b _{n}$, $n \in \mathbb N$, possesses a subsequence 
$b _{n} ^{\prime }$, $n \in \mathbb N$, such that $v(b _{n} ^{\prime 
}) > v(\varepsilon - 1) - (1/n)$, for each index $n$. It is 
therefore clear that the sequence $\pi _{n} = (\varepsilon - 1) 
^{p}/(pb _{n} ^{\prime })$, $n \in \mathbb N$, satisfies (3.2). 
Consider now the  polynomials $f _{n}(X)$, $g _{n}(X)$, $h _{n}(X)$, 
$t _{n}(X) \in K[X]$, $n \in \mathbb N$, defined by the rule $f 
_{n}(X) = X ^{p} - X - \pi _{n} ^{-1}$, $g _{n}(X) = \pi _{n} ^{p}f 
_{n}(X/\pi _{n}) = X ^{p} - \pi _{n} ^{p-1}X - \pi _{n} ^{p-1}$, $h 
_{n}(X) = (\varepsilon - 1) ^{-p}[((\varepsilon - 1)X + 1) ^{p} - 1 
- (\varepsilon - 1) ^{p}\pi _{n} ^{-1}]$ and $t _{n}(X) = \pi _{n} 
^{p}h(X/\pi _{n})$, for each $n$. It is easily verified that $t 
_{n}(X) \in O _{v}(K)[X]$, $t _{n}(X)$ is monic and the coefficients 
of the difference $t _{n}(X) – g _{n}(X)$ are divisible by 
$\varepsilon – 1$ (in $O _{v}(K)$). These observations enable one to 
deduce from (2.1) that there exists $N _{0} \in \mathbb N$, such 
that $f _{n}(X)$, $g _{n}(X)$, $h _{n}(X)$ and $t _{n}(X)$ are 
irreducible over $K$ and have roots in $F$, for each $n > N _{0}$. 
They also show that $N _{0}$ can be chosen so that $v _{F}(\eta _{n} 
^{-1}) = (1/p)v(\pi _{n})$ and $v _{F}(\psi (\eta _{n} ^{-1}) - \eta 
_{n} ^{-1}) = (2/p)v(\pi _{n}) < 2/(pn)$ whenever $n > N _{0}$, 
$\eta _{n} \in F$ and $f _{n}(\eta _{n}) = 0$. When both conditions 
hold, the sequence $\lambda _{n} = \eta _{n+N _{0}} ^{-1}$, $n \in 
\mathbb N$, satisfies (3.1), which proves Lemma \ref{lemm3.1} (c).
\end{proof}

\medskip
\begin{lemm}
\label{lemm3.2} Let $(E, w)$ be a Henselian field with
{\rm char}$(\widehat E) = q > 0$, $w(E) \neq qw(E)$ and {\rm 
Br}$(E ^{\prime }) _{q} = \{0\}$. Assume that $w(E)$ is 
Archimedean and $E ^{\prime } \in I(E _{\rm sep}/E)$ is the 
root field over $E$ of the binomial $X ^{q} - 1$. Then:
\par
{\rm (a)} $\widehat E$ is perfect, $w(E)/qw(E)$ is of order $q$, $q
\in P(E)$ and finite extensions of $E$ in $E(q)$ are totally
ramified; in particular, $q \notin P(\widehat E)$;
\par
{\rm (b)} For any cyclic extension $\Phi $ of $E$ in $E(q)$, there
exists $\Gamma _{0} \in I(E ^{\prime }(q)/E)$, such that $E ^{\prime
}(q)/\Gamma _{0}$ is a $\mathbb Z_{q}$-extension and $\Phi \cap \Gamma
_{0} = E$.
\end{lemm}
\par
\medskip
\begin{proof} The assertion that $q \in P(E)$ follows from the fact
that $(E, w)$ satisfies condition (2.4) (i). Since, by \cite{La},
Theorem~3.16, $E$ is a nonreal field, this assertion and \cite{Wh},
Theorem~2, indicate that $E(q)$ contains as a subfield a $\mathbb Z
_{q}$-extension $\Gamma $ of $E$; in particular, $[E(q)\colon E] =
\infty $. Let $L$ be a finite extension of $E$ in $E(q)$, and let
$[L\colon E] = q ^{k}$. It is clear from \cite{Ch2}, I, Lemma~4.2,
and the triviality of Br$(E) _{q}$ that $N(L/E) = E ^{\ast }$.
Hence, by the Henselity of $w$, $q ^{k}w(L) = w(E)$, which
implies in conjunction with (2.3) and the inequality $w(E) \neq
qw(E)$ that $\widehat \Phi = \widehat E$ and $w(E)/q ^{k}w(E)$ is a
cyclic group of order $q ^{k}$. These observations, combined with
(2.5) (iii), prove Lemma \ref{lemm3.2} (a). They also enable one to 
deduce from Galois theory the existence of a field $E _{1} \in 
I(E(q)/E)$, such that $[E _{1}\colon E] \le q$, $E _{1} \cap \Phi = 
E$ and $\Phi E _{1} = \Gamma E _{1}$. Clearly, $\Gamma E _{1}/E 
_{1}$ is a $\mathbb Z _{q}$-extension. For the rest of the proof of 
Lemma \ref{lemm3.2} (b), it suffices to observe that the set $Y(\Phi 
) = \{Y \in I(E ^{\prime }(q)/E _{1})\colon \ Y \cap \Phi = E\}$, 
partially ordered by inclusion, satisfies the conditions of Zorn's 
lemma, to take as $\Gamma _{0}$ any maximal element of $Y(\Phi )$, 
and using the projectivity of $\mathbb Z _{q}$ as a profinite group 
(cf. \cite{S1}, Ch. I, 5.9), to prove that $\Gamma _{0}\Phi _{1} = E 
^{\prime }(q)$ and $\mathcal{G}(\Gamma _{0}\Phi _{1}/\Gamma _{0}) 
\cong \mathbb Z _{q}$.
\end{proof}

\medskip
\begin{rema}
\label{rema3.3}
Retaining assumptions and notation as in Lemma \ref{lemm3.2}, put 
$\Gamma _{\ast } = E(q) \cap \Gamma _{0}$ and denote by $\Gamma 
_{n}$ the extension of $\Gamma _{0}$ in $E ^{\prime }(q)$ of degree 
$q ^{n}$, for each $n \in \mathbb N$. Observing that $E ^{\prime 
}/E$ is cyclic and $[E ^{\prime }\colon E] \mid (q - 1)$, one 
obtains that $E ^{\prime } (q)/E$ is Galois and $\Gamma _{0}$ 
contains a primitive $q$-th root of unity unless char$(E) = q$. Note 
further that $[\Gamma _{\ast }\colon E] = \infty $. Indeed, Lemma 
\ref{lemm3.2} (a) ensures that $\widehat E$ is an infinite perfect 
field (with $r _{q}(\widehat E) = 0$), so it it follows from 
\cite{Ch6}, Remark~4.2 and Lemma~4.3, that $r(q) _{E} = \infty $. By 
Galois theory, this means that there are infinitely many degree $q$ 
extensions of $E$ in $\Gamma _{\ast }$, whence $[\Gamma _{\ast 
}\colon E] = \infty $, as claimed. In addition, it follows from 
(2.3) and Lemma \ref{lemm3.2} (a) that $\widehat \Gamma _{\ast } = 
\widehat E$, $\widehat E ^{\prime }(q) = \widehat E ^{\prime }$, 
$w(E ^{\prime }\Gamma _{\ast }) = w(E ^{\prime }) + w(\Gamma _{\ast 
})$, $w(\Gamma _{\ast }) = qw(\Gamma _{\ast })$ and $w(E ^{\prime 
}\Gamma _{\ast }) = qw(E ^{\prime }\Gamma _{\ast })$. Observing also 
that $E ^{\prime }(q) = (E ^{\prime }\Gamma _{\ast }) (q) = \Gamma 
_{0}(q)$, one deduces from Lemma \ref{lemm3.1} that $E ^{\prime } 
(q)/\Gamma _{0}$ is an immediate quasiinertial $\mathbb Z 
_{q}$-extension.
\end{rema}
\par
\medskip
\section{Proof of Theorem \ref{theo1.2}}
\par
\medskip
Throughout this Section, we assume that $(\Phi , \omega )$ and $T$
satisfy the conditions of Theorem \ref{theo1.2}, and $\overline \Phi
$ is an algebraic closure of $\Phi _{\rm sep}$. Put $S(T) = \{p \in
\mathbb P\colon \ T _{p} \neq \{0\}\}$, $S _{q}(T) = S(T) \setminus 
\{q\}$, $S ^{\prime }(T) = \mathbb P \setminus S _{q}(T)$, and 
denote by $U$ the maximal extension of $\Phi $ in $\Phi _{\rm ur}$ 
whose finite subextensions have degrees not divisible by any $p \in S 
_{q}(T)$. The assumptions on $\Phi $, $\omega $ and $\widehat \Phi $ 
and the definition of $U$ indicate that $U/\Phi $ and $\Phi _{\rm 
ur}/U$ are Galois extensions with $\mathcal{G}(U/\Phi )$ and 
$\mathcal{G}(\Phi _{\rm ur}/U)$ isomorphic to the topological group 
products $\prod _{\pi ' \in S'(T)} \mathbb Z _{\pi '}$ and $\prod 
_{\pi_{q} \in S_{q}(T)} \mathbb Z _{\pi _{q}}$, respectively; this 
implies $q \notin \Pi (\widehat U)$, whence $\widehat U$ is 
infinite. As $\Phi $ is quasilocal, the obtained result proves (in 
conjunction with \cite{Ch2}, I, Proposition~4.4, Lemma~8.2 and 
Corollary~8.5) that Br$(U _{1}) _{\pi '} = \{0\}$, for every $U _{1} 
\in I(\overline \Phi /U)$ and each $\pi ^{\prime } \in S ^{\prime 
}(T)$. At the same time, it follows from (2.4) and the equality 
$\omega (U) = \omega (\Phi )$ that $\Phi (q) \notin I(U/\Phi )$, 
which ensures that $q \in P(U)$. Observing that $\omega _{U}$ is 
discrete and Henselian, one obtains from \cite{TY}, Proposition~2.2, 
that finite extensions of $U$ in $\Phi _{\rm sep}$ are defectless. 
Since $\widehat \Phi $ is perfect, $U$ does not possess inertial 
proper extensions in $U(q)$, and we have Br$(U _{1}) _{q} = \{0\}$, 
$U _{1} \in I(\overline \Phi /U)$, one also concludes that finite 
extensions of $U$ in $U(q)$ are totally ramified and 
$\mathcal{G}(U(q)/U)$ is a free pro-$q$-group (cf. \cite{S1}, Ch. I, 
4.2, and Ch. II,  ). Note further that $r(q) _{U} = \infty $; since 
$\omega _{U}$ is Henselian and discrete, and $\widehat U$ is 
infinite, this follows from \cite{Po}, (2.7) (as well as from Remark 
\ref{rema3.3} and the fact that Br$(U) _{q} = \{0\}$). The rest of 
our proof relies on the observation that the set $\Sigma $ of all 
$\Theta \in I(\Phi _{\rm sep}/U)$, such that $\Theta \cap \Phi _{\rm 
ur} = U$ and the degrees of finite extensions of $U$ in $\Theta $ 
are not divisible by $q$, is nonempty and satisfies the conditions 
of Zorn's lemma with respect to the partial ordering by inclusion. 
Fix a maximal element $\Theta ^{\prime } \in \Sigma $ and put 
$\omega ^{\prime } = \omega _{\Theta '}$. Then it follows from 
Galois theory, statement (2.3), the projectivity of 
$\mathcal{G}(\Phi _{\rm ur}/U)$ as a profinite group, and the 
triviality of the groups Br$(U _{1}) _{q}$, $U _{1} \in I(\overline 
\Phi /U)$, that $\Theta ^{\prime }$ satisfies the following:
\par
\medskip
(4.1) (i) $\Phi _{\rm ur}\Theta ^{\prime } = \Phi _{\rm tr}$; in 
particular, finite extensions of $U$ in $\Theta ^{\prime }$ are 
tamely totally ramified, $\omega ^{\prime }(\Theta ^{\prime }) \neq 
q\omega ^{\prime }(\Theta ^{\prime })$ and $\omega ^{\prime }(\Theta
^{\prime }) = p\omega ^{\prime }(\Theta ^{\prime })$, for each $p
\in \mathbb P \setminus \{q\}$.
\par
(ii) Finite extensions of $\Theta ^{\prime }$ in $\Theta ^{\prime
}(q)$ are totally ramified.
\par
(iii) $\mathcal{G}(\Theta ^{\prime }(q)/\Theta ^{\prime })$ is a
free pro-$q$-group, $r(q)_{\Theta '} = \infty $ and Br$(\Theta 
^{\prime \prime }) _{q} = \{0\}$, for every $\Theta ^{\prime \prime 
} \in I(\overline \Phi /\Theta ^{\prime })$.
\par
\medskip
The former assertion of (4.1) (iii) and \cite{Wh}, Theorem~2, imply 
the existence of a $\mathbb Z _{q}$-extension $\Gamma $ of $\Theta
^{\prime }$ in $\Phi _{\rm sep}$. Put $\Gamma _{0} = \Theta ^{\prime
}$, and for each $n \in \mathbb N$, let $\Gamma _{n}$ be the
extension of $\Theta ^{\prime }$ in $\Gamma $ of degree $q ^{n}$. It
follows from Galois theory and the assumption on $\widehat \Phi $
that the compositum $U ^{\prime } = \Theta ^{\prime }\Gamma \Phi
_{\rm ur}$ is a Galois extension of $\Theta ^{\prime }$ with
$\mathcal{G}(U ^{\prime }/\Theta ^{\prime }) \cong \prod _{\pi \in
S(T)} \mathbb Z _{\pi }$. This implies cd$(\mathcal{G}(U ^{\prime
}/\Theta ^{\prime })) = 1$, which means that $\mathcal{G}(U ^{\prime
}/\Theta ^{\prime })$ is a projective profinite group (cf. \cite{S1},
Ch. I, 4.2 and 5.9). Note also that the set $\widetilde \Sigma =
\{\widetilde \Theta \in I(\overline \Phi /\Theta ^{\prime })\colon \
\widetilde \Theta \cap U ^{\prime } = \Theta ^{\prime }\}$,
partially ordered by inclusion, satisfies the conditions of Zorn's
lemma. Let $\widetilde K$ be a maximal element of $\widetilde \Sigma
$, $\tilde v = \omega _{\widetilde K}$ and $\tilde k$ the residue
field of $(\widetilde K, \tilde v)$. It is easily verified that
$\widetilde K$ and $\tilde k$ are perfect fields, and it follows
from the projectivity of $\mathcal{G}(U ^{\prime }/\Theta ^{\prime
})$ that $\overline \Phi = U ^{\prime }\widetilde K$. Hence, by
Galois theory and the equality $\widetilde K \cap U ^{\prime } =
\Theta ^{\prime }$, $\mathcal{G}_{\widetilde K} \cong \mathcal{G}(U
^{\prime }/\Theta ^{\prime })$. Our argument, together with the
former part of (4.1) (iii), also proves that there exists a $\mathbb
Z _{q}$-extension of $\Theta ^{\prime }$ in $\widetilde K$. Since
$\omega $ is discrete, this enables one to deduce the former part of
the following assertion from (4.1) (i), (ii) and (2.3):
\par
\medskip
(4.2) $\tilde v(\widetilde K) = \mathbb Q$, $\tilde k/\widehat \Phi
$ is an algebraic extension and $\Gamma \widetilde K/\widetilde K$
is immediate. Moreover, $\widetilde K(q) = \Gamma \widetilde K$,
$\Gamma \widetilde K/\widetilde K$ is a $\mathbb Z _{q}$-extension
with $[\Gamma _{n}\widetilde K\colon \Gamma _{n-1}\widetilde K] =
q ^{n}$, for each $n \in \mathbb N$, and $\Phi _{\rm ur}\widetilde
K(q)/\Phi _{\rm ur}\widetilde K$ is a quasiinertial $\mathbb Z
_{q}$-extension.
\par
\medskip\noindent
As $\Gamma /\Theta ^{\prime }$ is a $\mathbb Z _{q}$-extension, 
$\widetilde K \cap U ^{\prime } = \Theta ^{\prime }$, and $\Phi 
_{\rm ur}$ contains a primitive $q$-th root of unity unless 
char$(\Phi ) = q$, the latter part of (4.2) follows at once from the 
former one, Galois theory and Lemma \ref{lemm3.1} (c). Taking into 
account that the degrees of finite extensions of $\widetilde K$ in 
$\Phi _{\rm ur}\widetilde K$ are not divisible by $q$ 
($\mathcal{G}(\Phi _{\rm ur}\widetilde K/\widetilde K) \cong 
\mathcal{G}(\Phi _{\rm ur}/U) \cong \prod _{\pi _{q} \in S_{q}(T)} 
\mathbb Z _{\pi _{q}}$), and using trace transitivity in towers of 
finite separable extensions, one concludes that (4.2) can be 
supplemented as follows:
\par
\medskip
(4.3) The $\mathbb Z _{q}$-extension $\Gamma \widetilde K/\widetilde 
K$ is quasiinertial.
\par
\medskip\noindent
We are now in a position to construct a quasilocal Henselian field of 
the type required by Theorem \ref{theo1.2}. Fix a positive number 
$\gamma \in \mathbb R \setminus \mathbb Q$ and a rational function 
field $\widetilde K(X)$ in one indeterminate over $\widetilde K$. It 
is easily verified that $\tilde v$ is uniquely extendable to a 
valuation $\tilde v _{\gamma }$ of $\widetilde K(X)$ satisfying the equality $\tilde v _{\gamma }(X) = \gamma $, and it follows from the choice of $\gamma $ that $\tilde v _{\gamma }(\widetilde K(X))$ is an Archimedean group equal to the sum of $\mathbb Q$ and $\langle \gamma \rangle $. In addition, it becomes clear that $\tilde v _{\gamma 
}(\widetilde K(X))$ is isomorphic (as an abstract group) to the 
direct sum $\mathbb Q \oplus  \langle \gamma \rangle $, and the 
residue field of $(\widetilde K(X), \tilde v _{\gamma })$ coincides 
with $\tilde k$. Note also that $\bar v _{\gamma }(\overline \Phi 
(X)) = \tilde v _{\gamma }(\widetilde K(X))$, where $\bar v _{\gamma 
}$ is the valuation of $\overline \Phi (X)$ naturally extending 
$\tilde v _{\overline \Phi }$ and $\tilde v _{\gamma }$. Now take a Henselization $(K, v)$ of $(\widetilde K(X), \tilde v _{\gamma })$ so that $K \subset \overline \Phi (X) _{\rm sep}$, and fix an algebraic closure $\overline K$ of $K$ including $\overline \Phi (X) _{\rm 
sep}$ as a subfield. It is well-known that $(K, v)/(\widetilde K(X), \tilde v _{\gamma })$ is immediate. The obtained properties of 
$\tilde v _{\gamma }(\widetilde K(X))$ and the equality $\tilde v _{\gamma }(\widetilde K(X)) = v(K)$ indicate that $v(K)/pv(K)$ is of order $p$ and $v(\gamma ) \notin pv(K)$, for any $p \in \mathbb P$; 
in particular, $v(K)$ is totally indivisible. We show that $K$, $v$ 
and $I _{\infty } = \Gamma K$ are admissible by Theorem 
\ref{theo1.2}. As a first step towards this, we prove the following:
\par
\medskip
(4.4) (i) $\widetilde K$ is algebraically closed in $K$ and
$\overline \Phi K/K$ is a Galois extension with
$\mathcal{G}(\overline \Phi K/K) \cong \mathcal{G}_{\widetilde K}
\cong \prod _{p \in S(T)} \mathbb Z _{p}$; in addition, $v(\overline
\Phi K) = v(K)$, $\Gamma K/K$ is an immediate $\mathbb Z
_{q}$-extension, and $[\Gamma _{n}K\colon K] = q ^{n}$, for each $n 
\in \mathbb N$;
\par
(ii) $\Gamma \Omega /\Omega $ is a quasiinertial $\mathbb Z
_{q}$-extension, for every finite extension $\Omega $ of $K$ in $\overline K$.
\par
\medskip\noindent
Let $K(\sqrt[q]{X})$ be an extension of $K$ in $\overline K$ obtained 
by adjunction of a $q$-th root of $X$. It is clear from the 
definition of $\tilde v _{\gamma }$ and the immediacy of the valued extension $(K, v)/(\widetilde K(X), \tilde v _{\gamma })$ that $K(\sqrt[q]{X})/K$ is totally ramified and $[K(\sqrt[q]{X})\colon K] 
= q$. Since $\widetilde K$ is perfect and $K \in I(\widetilde K(X) 
_{\rm sep}/\widetilde K(X))$, it is also clear that in case 
char$(\Phi ) = q$, $K(\sqrt[q]{X})$ is the unique purely inseparable extension of $K$ in $\overline K$ of degree $q$. Note further that 
the inclusion of $v(K) = \tilde v _{\gamma }(\widetilde K(X))$ in $\mathbb R$ guarantees that $\widetilde K(X) _{\tilde v _{\gamma }}$ 
is Henselian with respect to its valuation $v _{\gamma }$ 
continuously extending $\tilde v _{\gamma }$. As $(\widetilde K(X) _{\tilde v _{\gamma }}, v _{\gamma })$ is immediate over $(\widetilde 
K(X), \tilde v _{\gamma })$, these facts show that $K$ is $\widetilde K(X)$-isomorphic to the (relative) algebraic closure of $\widetilde 
K(X)$ in $\widetilde K(X) _{\tilde v _{\gamma }}$ (cf. \cite{E3}, 
Sect. 18.3). At the same time, it follows from the definition of the 
valuation $\bar v _{\gamma }$ of $\overline \Phi (X)$ that an element $\rho \in \overline \Phi $ lies in $\widetilde K(X) _{v _{\gamma }}$ 
if and only if $\rho \in \widetilde K _{\tilde v}$. Taking also into account that $\widetilde K$ is algebraically closed in $\widetilde K _{\tilde v}$ (because $\widetilde K$ is perfect and $\tilde v$ is Henselian), one concludes that $\widetilde K$ is algebraically closed 
in $K$. In view of Galois theory, this means that $\overline \Phi 
K/K$ is a Galois extension with $\mathcal{G}(\overline \Phi K/K) 
\cong \mathcal{G}_{\widetilde K}$. These observations prove the 
former part of (4.4) (i), so we turn to the proof of the latter one. Using the equalities $\bar v _{\gamma }(\overline \Phi (X)) = v 
_{\gamma }(\widetilde K(X)) = v(K)$, and replacing $\widetilde K$ by 
any of its finite extensions in $\overline \Phi $, one obtains 
further that $v(\overline \Phi K) = v(K)$. As cd$_{p'}(\mathcal{G} _{\tilde k}) = 0$, for every $p ^{\prime } \in \mathbb P \setminus 
S(T)$, this result implies in conjunction with (2.3) and (4.2) that $\Gamma K/K$ is immediate and $\Gamma \cap K = \Theta ^{\prime }$, so (4.4) (i) is proved. As to (4.4) (ii), it can be deduced from Galois theory and Lemma \ref{lemm2.1}, since $\Gamma \widetilde K/\widetilde 
K$ is quasiinertial (by (4.3)), $v(K) \le \mathbb R$, $v$ extends 
$\tilde v$ upon $K$, $v(K)$ is Archimedean and $\widetilde K$ is 
algebraically closed in $K$.
\par
Next we show that Br$(K) _{p} \neq \{0\}$ if and only if $p \in
S(T)$. Suppose first that $p \notin S(T)$. Then $p \dagger
[\widetilde M\colon \widetilde K]$, for any finite extension
$\widetilde M$ of $\widetilde K$, which implies Br$(K) _{p} \cap
{\rm Br}(\overline \Phi K/K) = \{0\}$ (cf. \cite{P}, Sect. 13.4). On
the other hand, $\overline \Phi K/\overline \Phi $ is a field
extension of transcendency degree $1$, so it follows from Tsen's
theorem (see \cite{P}, Sect. 19.4) that Br$(\overline \Phi K) =
\{0\}$. It is therefore easy to see that Br$(K) = {\rm Br}(\overline
\Phi K/K)$ and Br$(K) _{p} = \{0\}$. Assume now that $p \in S(T)$.
Then it follows from Galois theory and (4.4) that $I(\overline \Phi
K/K)$ contains a cyclic extension $Y _{p}$ of $K$ of degree $p$.
Moreover, by (4.4) (i), $v(Y _{p}) = v(K)$, whence the uniqueness of
$v _{Y _{p}}$ implies $N(Y _{p}/K) \subseteq \{\lambda \in K ^{\ast
}\colon \ v(\lambda ) \in pv(K)\}$. Since $v(K) \neq pv(K)$, this
means that Br$(Y _{p}/K) \neq \{0\} \neq {\rm Br}(K) _{p}$.
\par
It remains to be proved that $K$ is quasilocal and Br$(K) \cong T$.
Assuming as above that $p \in S(T)$, let $G _{p}$ be a Sylow
pro-$p$-subgroup of $\mathcal{G}_{K}$ and $K _{p}$ the fixed field
of $G _{p}$. We show that $K _{p}$ is $p$-quasilocal with Br$(K
_{p}) \cong \mathbb Z(p ^{\infty })$. The equality $v(K) = v
_{\gamma }(\widetilde K(X))$ and the isomorphism $v(K _{p})/pv(K
_{p}) \cong v(K)/pv(K)$ guarantee that $v(K _{p})/pv(K _{p})$ is of
order $p$. When $p \neq q$, this enables one to deduce from (4.4)
and \cite{E2}, Lemma~1.2, that $K _{p} ^{\ast }/K _{p} ^{\ast p}$ is
a group of order $p ^{2}$. As $K _{p}$ contains a primitive $p$-th
root of unity and Br$(K) _{p} \cap {\rm Br}(K _{p}/K) = \{0\}$, the
obtained results and Galois cohomology (see \cite{W}, Lemma~7,
\cite{MS}, (11.5), and \cite{S1}, Ch. I, 4.2) prove that $G _{p}$ is
a Demushkin group, $r(p)_{K _{p}} = 2$ and Br$(K _{p}) \cong \mathbb
Z(p ^{\infty })$. Hence, by \cite{Ch2}, I, Lemma~3.8, $K _{p}$ is
$p$-quasilocal. It remains to be seen that $K _{q}$ is
$q$-quasilocal and Br$(K _{q}) \cong \mathbb Z(q ^{\infty })$. As
$\tilde k$ is perfect, cd$_{q}(\mathcal{G}_{\tilde k}) = 0$ and
$\widehat K = \tilde k$, $\widehat K _{q}$ is an algebraic closure
of $\tilde k$, so $\widehat Z = \widehat K _{q}$, for each $Z \in
I(K _{\rm sep}/K _{q})$. In addition, it follows from Tsen's theorem
that Br$(K _{q}) = {\rm Br}(\Gamma K _{q}/K _{q})$. Applying (4.4),
(2.8) and Lemma \ref{lemm2.1}, one also sees that $\nabla
_{0}(\Gamma _{1}) \subseteq N(\Gamma _{n}K _{q}/\Gamma _{1}K _{q})$,
for each $n \in \mathbb N$. As $\Gamma _{1}K _{q}/K _{q}$ is
immediate, this enables one to deduce from (2.7) and Hilbert's
Theorem 90 that an element $\theta \in K _{q} ^{\ast }$ lies in
$N(\Gamma _{\nu }K _{q}/\Gamma _{1}K _{q})$, for a given index $\nu
$, if and only if $\theta ^{q} \in N(\Gamma _{\nu }K _{q}/K _{q})$.
Since Br$(\Gamma K _{q}/K _{q}) = \cup _{n=1} ^{\infty } {\rm
Br}(\Gamma _{n}K _{q}/K _{q})$, these observations and the canonical
isomorphisms Br$(\Gamma _{n}K _{q}/K _{q}) \cong K _{q} ^{\ast
}/N(\Gamma _{n}K _{q}/K _{q})$, $n \in \mathbb N$ (cf. \cite{P},
Sect. 15.1, Proposition~b), prove that $_{q}{\rm Br}(K _{q}) = {\rm
Br}(\Gamma _{1}K _{q}/K _{q})$. The obtained result, combined with
the fact that $\widehat K _{q}$ is algebraically closed and $v(K
_{q})/qv(K _{q})$ is of order $q$, proves that $N(\Gamma _{1}K
_{q}/K _{q}) = \{\mu \in K _{q} ^{\ast }\colon \ v(\mu ) \in qv(K
_{q})\}$, $_{q} {\rm Br}(K _{q})$ is of order $q$ and Br$(K _{q})
\cong \mathbb Z(q ^{\infty })$. Let now $\Lambda $ be an extension
of $K _{q}$ in $K _{\rm sep}$, such that $[\Lambda \colon K _{q}] =
q$ and $\Lambda \neq \Gamma _{1}K _{q}$, and let $V _{q}(\Lambda ) =
\{\lambda \in \Lambda \colon \ v _{\Lambda }(\lambda ) \in
qv(\Lambda )\}$. Applying (4.4) and (2.7), and arguing as in the
proof of the isomorphism Br$(K _{q}) \cong \mathbb Z(q ^{\infty })$,
one obtains consecutively the following results:
\par
\medskip
(4.5) (i) $V _{q} (\Lambda ) \subseteq N(\Gamma _{1}\Lambda /\Lambda
)$; $\tau (\lambda ^{\prime })\lambda ^{\prime -1} \in N(\Gamma
_{1}\Lambda /\Lambda )$, for each $\lambda ^{\prime } \in \Lambda
^{\ast }$ and every generator $\tau $ of $\mathcal{G}(\Lambda /K
_{q})$;
\par
(ii) Br$(\Gamma _{1}\Lambda /\Lambda ) =$ $_{q} {\rm Br}(\Lambda )
\neq \{0\}$; hence $N(\Gamma _{1}\Lambda /\Lambda ) \neq \Lambda
^{\ast }$.
\par
\medskip\noindent
As $\widehat \Lambda $ is algebraically closed and $v(\Lambda
)/qv(\Lambda )$ has order $q$, one also proves that
\par
\medskip
(4.6) (i) $N(\Gamma _{1}\Lambda /\Lambda ) = V _{q}(\Lambda )$ and
$\Gamma _{1}\Lambda /\Lambda $ is immediate.
\par
(ii) $K ^{\ast } \subseteq N(\Gamma _{1}\Lambda /\Lambda )$,
provided that $\Lambda $ is totally ramified over $K _{q}$; when
this holds, Br$(\Gamma _{1}/K _{q}) \subseteq {\rm Br}(\Lambda /K
_{q}) =$ $_{q} {\rm Br}(K _{q})$.
\par
\medskip\noindent
In view of (4.5) (ii) and (4.6) (ii), it suffices, for the proof of
the $q$-quasilocality of $K _{q}$, to show that $\Lambda /K _{q}$ is
totally ramified. Assuming the opposite, one gets from (2.3) and the
equality $\widehat \Lambda = \widehat K _{q}$ that $\Lambda /K _{q}$
is immediate. Fix a generator $\tau $ of $\mathcal{G}(\Lambda /K
_{q})$, denote by $\tau ^{\prime }$ the $\Gamma _{1}$-automorphism
of $\Gamma _{1}\Lambda $ extending $\tau $, and put $D _{\rho } =
(\Lambda /K _{q}, \tau , \rho )$, $\Delta _{\rho } = (\Gamma
_{1}\Lambda /\Gamma _{1}, \tau ^{\prime }, \rho )$, for some $\rho
\in K _{q} ^{\ast }$. Clearly, $\Delta _{\rho } \cong D _{\rho }
\otimes _{K _{q}} \Gamma _{1}$ over $\Gamma _{1}$. Hence, the
equality Br$(\Gamma _{1}/K _{q}) =$ $_{q} {\rm Br}(K _{q})$ requires
that $[\Delta _{\rho }] = 0$ in Br$(\Gamma _{1})$. On the other
hand, (4.6) (i) and the assumption on $\Lambda /K _{q}$ imply
$\Gamma _{1}\Lambda /\Gamma _{1}$ is immediate. This shows that if
$v(\rho ) \notin qv(K _{q})$, then $D _{\rho } \in d(K _{q})$ and
$\Delta _{\rho } \in d(\Gamma _{1})$, whence $[\Delta _{\rho }] \neq
0$. The observed contradiction proves that $\Lambda /K _{q}$ is
totally ramified, so $K _{q}$ is $q$-quasilocal (with Br$(K _{q})
\cong \mathbb Z(q ^{\infty })$).
\par
It is now easy to complete the proof of Theorem \ref{theo1.2}.
Indeed, it follows from \cite{Ch2}, I, Lemma~8.3, and the
$p$-quasilocal property of the fields $K _{p}$, $p \in \Pi (K)$,
that $K$ is quasilocal. As $K$ is nonreal and $S(T) = \{p \in
\mathbb P\colon \ {\rm Br}(K) _{p} \neq \{0\}\}$, this result,
\cite{Ch3}, Lemma~3.3 (i) (see also \cite{Ch2}, I, Theorem~3.1), and
the isomorphisms Br$(K _{p}) \cong \mathbb Z(p ^{\infty })$, $p \in
S(T)$, yield Br$(K) \cong T$. Theorem \ref{theo1.2} is proved.
\par
\medskip
\section{\bf Complements to Theorem \ref{theo1.2}}
\par
\medskip
First we show that, in residual characteristic $2$, Theorem
\ref{theo1.2} and \cite{Ch3}, Theorem~1.1, fully describe the
isomorphism classes of Brauer groups of quasilocal Henselian fields
admissible by Proposition \ref{prop1.1}.
\par
\medskip
\begin{prop}
\label{prop5.1}
Let $(K, v)$ be a quasilocal Henselian field satisfying the
conditions of Proposition \ref{prop1.1}, and let {\rm
char}$(\widehat K) = 2$. Then there exists an immediate
norm-inertial $\mathbb Z _{2}$-extension $\Gamma /K$; in particular,
{\rm Br}$(K) _{2} \cong \mathbb Z(2 ^{\infty })$.
\end{prop}
\par
\medskip
\begin{proof}
Proposition \ref{prop1.1} and our assumptions show that $\widehat K$ 
is perfect and cd$_{2}(\mathcal{G}_{\widehat K}) = 0$. In view of 
(2.3) and (2.5), this ensures that cd$_{2}(\mathcal{G}(K _{\rm 
tr}/K)) = 0$, $K _{\rm tr}$ is the fixed field of a Sylow 
pro-$2$-subgroup of $\mathcal{G}_{K}$, and $K _{\rm tr}$ has a 
$\mathbb Z _{2}$-extension $Y$ in $K _{\rm sep}$. In addition, it 
follows from the uniqueness of $Y$ and the normality of $K _{\rm 
tr}/K$ that $Y/K$ is a Galois extension. Note also that 
$\mathcal{G}(Y/K _{\rm tr}) \cong \mathbb Z _{2}$ and 
$\mathcal{G}(Y/K _{\rm tr})$ is a normal Sylow pro-$2$-subgroup of 
$\mathcal{G}(Y/K)$. These observations indicate that 
$\mathcal{G}(Y/K _{\rm tr})$ is included in the centre of 
$\mathcal{G}(Y/K)$. It is therefore clear from Galois theory and 
Burnside's theorem (cf. \cite{Ha}, Theorem~14.3.1, and \cite{S1}, 
Ch. I, 5.9) that $\mathcal{G}(Y/K)$ possesses a closed normal 
subgroup $N$, such that $\mathcal{G}(Y/K _{\rm tr})N = 
\mathcal{G}(Y/K)$ and $\mathcal{G}(Y/K _{\rm tr}) \cap N = \{1\}$.
This means that $\mathcal{G}(Y/K) \cong \mathcal{G}(Y/K _{\rm tr})
\times N$, the fixed field $\Gamma $ of $N$ is a $\mathbb Z
_{2}$-extension of $K$, $\Gamma K _{\rm tr} = Y$ and $\Gamma \cap K
_{\rm tr} = K$. As $Y/K _{\rm tr}$ is immediate and finite 
extensions of $K$ in $K _{\rm tr}$ are of odd degrees, one deduces 
from (2.3) and Proposition \ref{prop1.1} that $\Gamma /K$ is 
immediate and norm-inertial. Hence, by \cite{Ch3}, Theorem~1.1, 
Br$(K) _{2} \cong \mathbb Z(2 ^{\infty })$.
\end{proof}

\par
\medskip
Theorem \ref{theo1.2} and Proposition \ref{prop5.1} can be
complemented as follows:
\par
\medskip
\begin{prop}
\label{prop5.2}
Let $(\Phi , \omega )$ satisfy the conditions of Theorem
\ref{theo1.2}, for some $q > 2$, and let $T _{q}$ be a divisible
subgroup of $\mathbb Q/\mathbb Z$ with $T _{q} = \{0\}$. Then there
exists a valued extension $(K, v)$ of $(\Phi , \omega )$, such that
$v$ is Henselian, $v(K)$ is totally indivisible and Archimedean,
$\widehat K = \widehat \Phi _{\rm sep}$, $K/\Phi $ has transcendency
degree $1$, {\rm Br}$(K) \cong T$, and $K$ admits a defectful finite
extension in $K _{\rm sep}$.
\end{prop}
\par
\medskip
\begin{proof}
It is clearly sufficient to consider only the special case where
char$(\Phi ) = q$ or $\Phi $ contains a primitive $q$-th root of
unity. Our argument goes along the same lines as the proof of
Theorem \ref{theo1.2}, so we omit the details and note only its main
steps. Our starting point are the following statements:
\par
\medskip
(5.1) For any integer $m \ge 2$ dividing $q - 1$, $\Phi _{\rm sep}$
contains as a subfield a totally ramified Galois extension $\Psi
_{m}$ of $\Phi $, such that $[\Psi _{m}\colon \Phi ] = qm$ and
$\mathcal{G}(\Psi _{m}/\Phi )$ is a nonabelian metacyclic group. For
example, if char$(\Phi ) = q$ and $\pi $ is a generator of $M
_{\omega }(\Phi )$, then one may take as $\Psi _{m}$ the field $\Phi
(\xi _{m})$, where $\xi _{m} \in \Phi _{\rm sep}$ is a root of the
polynomial $g _{m}(X) = (X ^{q} - X) ^{m} - \pi ^{-1}$. When 
char$(\Phi ) = 0$, $\Psi _{m}$ can be chosen among the subfields of 
the root field $\widetilde \Psi _{m} \in I(\overline \Phi /\Phi )$ 
of the polynomial $h _{m}(X) = (X ^{q} - 1) ^{m} - \pi $, under the 
same hypothesis on $\pi $.
\par
\medskip\noindent
Fix $m$ as in (5.1), put $\Psi _{m} ^{\prime } = \Psi _{m}K _{\rm
tr}$ and let $\widetilde K \in I(\overline \Phi /\Phi _{\rm ur})$ be
maximal with respect to the property that $\widetilde K \cap \Psi
_{m} ^{\prime } = \Phi _{\rm ur}$. Observing that $\mathcal{G}(\Psi
_{m} ^{\prime }/\Phi _{\rm ur})$ is a pro-supersolvable group and
$\mathcal{G}(\Psi _{m} ^{\prime }\widetilde K/\widetilde K) \cong
\mathcal{G}(\Psi _{m}/\Phi _{\rm ur})$, and applying Galois theory
and Huppert's theorem (cf. \cite{Ha}, Theorem~10.5.8), one obtains
that $\mathcal{G}_{\widetilde K}$ is pro-supersolvable. Therefore, by \cite{Ha}, Theorem~10.5.1, $\mathcal{G}(L/\widetilde K)$ is supersolvable, for each finite Galois extension $L/\widetilde K$. 
Hence, for any $p \in \mathbb P$, $L$ possesses a subfield $L _{[p]}$ that is a Galois extension of $\widetilde K$ with $\mathcal{G}(L _{[p]}/\widetilde K)$ isomorphic to a Hall $\Pi $-subgroup of $\mathcal{G}(L/\widetilde K)$, where $\Pi = \{\pi \in \mathbb P\colon 
\ \pi \le p\}$ (cf. \cite{Ha}, Sect. 9.3 and Corollary~10.5.2). Let 
$H _{p}$ be a Sylow $p$-subgroup of $\mathcal{G}(L _{[p]}/\widetilde 
K)$. We show that $H _{p}$ is cyclic. The group $\mathcal{G}(L _{[p]}/\widetilde K)$ is supersolvable which implies that it includes 
$H _{p}$ as a normal subgroup. The Frattini subgroup $\Phi (H _{p})$ 
of $H _{p}$ is characteristic in $H _{p}$, so it is normal in $\mathcal{G}(L _{[p]}/\widetilde K)$, and by Galois theory, the fixed field, say $\Lambda _{p}$, of $\Phi (H _{p})$ is a Galois extension 
of $\widetilde K$. Let $\overline H _{p}$ be a Sylow $p$-subgroup of $\mathcal{G}(\Lambda _{p}/\widetilde K)$. Then $\mathcal{G}(\Lambda _{p}/\widetilde K) \cong \mathcal{G}(L _{[p]}/\widetilde K)/\Phi (H _{p})$, $\overline H _{p} \cong H _{p}/\Phi (H _{p})$ and $\overline 
H _{p}$ is an abelian normal subgroup of $\mathcal{G}(\Lambda _{p}/\widetilde K)$ of period $p$. Also, $\mathcal{G}(\Lambda _{p}/\widetilde K)$ is supersolvable, and by \cite{Ha}, 
Corollary~10.5.2, it has a normal subgroup of order $p$, the greatest prime divisor of $[\Lambda _{p}\colon \widetilde K]$. Regarding $\overline H _{p}$ as an $\mathbb F _{p}$-vector space, and 
considering the action on $\overline H _{p}$ by conjugation of some 
Hall $\Pi _{p}$-subgroup of $\mathcal{G}(\Lambda _{p}/\widetilde K)$, 
for $\Pi _{p} = \Pi \setminus \{p\}$, one obtains from Maschke's 
theorem that if $\overline H _{p}$ is noncyclic, then it decomposes 
into the direct product of normal subgroups of $\mathcal{G}(\Lambda 
_{p}/\widetilde K)$ of order $p$. In view of Galois theory, this 
leads to the conclusion that if $\overline H _{p}$ is noncyclic, 
then there exist degree $p$ extensions $\Lambda _{1}$ and $\Lambda 
_{1} ^{\prime }$ of $\widetilde K$ in $\Lambda _{p}$, such that 
$[\Lambda _{1}\Lambda _{1} ^{\prime }\colon \widetilde K] = p ^{2}$. 
Therefore, $\Lambda _{1}$ and $\Lambda _{1} ^{\prime }$ are not 
$\widetilde K$-isomorphic. Our conclusion, however, contradicts the 
maximum condition on $\widetilde K$ and so proves that $\overline H 
_{p}$ is cyclic. It is now easy to see that $H _{p}$ has a unique 
maximal subgroup, whence it is cyclic as well. Summing-up the 
obtained results, one proves the following:
\par
\medskip
(5.2) $\widetilde K$ is perfect, $\widetilde K _{\rm tr} = \Phi
_{\rm tr}\widetilde K$ and $\mathcal{G}(\widetilde K _{\rm
tr}/\widetilde K) \cong \mathcal{G}(\Phi _{\rm tr}/\Phi _{\rm ur})$;
$\overline \Phi /\widetilde K _{\rm tr}$ is a quasiinertial $\mathbb
Z _{q}$-extension and the Sylow pro-$p$-subgroups of
$\mathcal{G}_{\widetilde K}$ are isomorphic to $\mathbb Z _{p}$, for
each $p \in \mathbb P$; the Sylow pro-$q$-subgroup of
$\mathcal{G}_{\widetilde K}$ is normal and equals the closure of the
commutator subgroup of $\mathcal{G}_{\widehat K}$.
\par
\medskip\noindent
As in the proof of Theorem \ref{theo1.2}, let $\widetilde K(X)$ be
the rational function field in an indeterminate $X$ with coefficients
in $\widetilde K$, and $\tilde v _{\gamma }$ the valuation of 
$\widetilde K(X)$ extending $\omega $ so that $\tilde v _{\gamma }(X) 
= \gamma $, where $\gamma $ is a given element of $\mathbb R 
\setminus \mathbb Q$. Fix a Henselization $(K _{0}, v _{0})$ of $(\widetilde K(X), \tilde v _{\gamma })$, put $N(T) = \{p \in \mathbb 
P \setminus \{q\}\colon \ T _{p} = \{0\}\}$, and let $K$ be an 
extension of $K _{0}$ in $K _{0,{\rm sep}}$ maximal with respect to 
the property that $K \cap \overline \Phi = \widetilde K$ and finite extensions of $K _{0}$ in $K$ are cyclic and totally ramified of 
degrees not divisible by any $p ^{\prime } \in \mathbb P \setminus 
N(T)$. Note that $(K, v)$ has the properties required by Proposition \ref{prop5.2}, where $v$ is a prolongation of $v _{0}$ on $K$. It 
follows from the definition of $(K, v)$ that $\widehat K = \widehat 
\Phi _{\rm sep}$, and for each $p \in N(T)$, $v(K)/pv(K)$ is of order 
$p$ and $K(p)/K$ is a $\mathbb Z _{p}$-extension; also, $v(K)/pv(K)$ 
has order $p ^{2}$ and $r _{p}(K) = 2$ in case $p \in \mathbb P 
\setminus N(T)$ and $p \neq q$. Since $K$ contains a primitive $m$-th root of unity, for any $m \in \mathbb N$ not divisible by $q$, these observations show that Br$(K) _{p} = \{0\}$, $p \in \mathbb N(T)$, 
and Br$(K) _{p} \cong \mathbb Z(p ^{\infty })$, $p \in \mathbb P 
\setminus N(T)$, $p \neq q$. The assertion that $K$ is quasilocal and satisfies with $v$ the conditions of Proposition \ref{prop1.1} is 
proved similarly to Theorem \ref{theo1.2}, so what remains to be seen 
is that Br$(K) _{q} = \{0\}$. Let $\Theta _{m}$ be the extension of 
$\Phi $ in $\Psi _{m}$ of degree $m$. Then $\Theta _{m}/\Phi $ and $\Theta _{m}K/K$ are cyclic extensions of degree $m$, and $\Theta 
_{m}K$ has an immediate $\mathbb Z _{q}$-extension $\Gamma $ in $K 
_{\rm sep}$ that is a Galois extension of $K$. This implies 
Br$(\Theta _{m}K) _{q} \cong \mathbb Z(q ^{\infty })$. Using (5.1) 
and regarding $_{q}{\rm Br}(\Theta _{m}K)$ as a module over the group algebra $\mathbb F _{q}[\mathcal{G}(\Theta _{m}K/K)]$, one obtains 
that if $\tau _{m}$ is a generator of $\mathcal{G}(\Theta _{m}K/K)$, 
then $\tau _{m}b = fb$, $b \in $ $_{q}{\rm Br}(\Theta _{m}K)$, for 
some $f \in \mathbb F _{q} ^{\ast }$, $f \neq 1$. As $m \mid q - 1$ 
and $m > 1$, this observation shows that $_{q}{\rm Br}(\Theta _{m}K)$ 
is included in the kernel of the corestriction homomorphism 
Br$(\Theta _{m}K) \to {\rm Br}(K)$, which enables one to deduce from 
the basic restriction-corestriction formula for Brauer groups (cf. \cite{Ti}) that Br$(K) _{q} = \{0\}$. Proposition \ref{prop5.2} is proved.
\end{proof}
\par
\medskip
\begin{rema}
\label{rema5.3}
Suppose that $(\widetilde K, \tilde v)$ and $\widetilde K(X)$ are 
defined as in the proof of Theorem \ref{theo1.2}, and $v _{0}$ is a 
restricted Gauss valuation of $\widetilde K(X)$ extending $\tilde v$ 
(see \cite{E3}, Example~4.3.2). Then, by \cite{Ch3}, 
Proposition~6.5, there exists a quasilocal Henselian field $(K, v)$,
such that $K \in I(\widetilde K(X) _{\rm sep}/\widetilde K)$, $v$ is
a prolongation of $v _{0}$, $\widehat K = \widehat K _{\rm sep} \neq
\widehat K ^{q}$, $K _{\rm sep} = K(q)$ and Br$(K)$ is a divisible 
hull of the (infinite) quotient group $\widehat K ^{\ast }/\widehat 
K ^{\ast q}$. As demonstrated at the end of \cite{Ch3}, Sect. 6, for 
any global field $\Psi $ of characteristic zero or $q$, this enables 
one to find (by the method of proving Theorem \ref{theo1.2}) field extensions $K _{t}/\Psi $, $t \in \mathbb N$, such that $K _{t}$ is quasilocal, $\mathcal{G}_{K _{t}}$ is a pro-$q$-group, the 
transcendency degree of $K _{t}/\Psi $ is equal to $t$, the class 
$d(K _{t}) \setminus \{K _{t}\}$ consists of division algebras of 
infinite genus, in the sense of \cite{CheRapRap}, and $[K _{t}\colon 
K _{t} ^{q}] = q$ in case char$(\Psi ) = q$. More precisely, by 
\cite{Ch2}, I, Corollaries~8.5 and 8.6, the genus of any $D _{t} \in 
d(K _{t})$ equals the set $\{[D _{t} ^{\prime }] \in {\rm Br}(K 
_{t})\colon \ D _{t} ^{\prime } \in d(K _{t}), {\rm ind}(D 
_{t} ^{\prime }) = {\rm ind}(D _{t})\}$, and also, the equivalence 
class of $D _{t}$, in the sense of \cite{KrMcKinn}, Definition~2.1. 
When char$(\Psi ) = 0$, these results ensure that $\mathcal{G}_{K 
_{t}}$ is a pro-$q$-group of Demushkin type with an infinite 
(continuous) cohomology group $H ^{2}(\mathcal{G}_{K _{t}}, \mathbb 
F _{q})$ (see \cite{Ch2}, I, Lemma~3.8, and \cite{Ch4}, 
Proposition~5.1).
\end{rema}
\par
\vskip0.4truecm
{\bf Acknowledgements.} In 2010 this research was partially supported
by the Institute of Mathematics of the Romanian Academy.

\par\medskip
Institute of Mathematics and Informatics, Bulgarian Academy of Sciences
\par
Acad. G. Bonchev Str., bl. 8, 1113 Sofia, Bulgaria
\end{document}